\documentclass{article}
\usepackage{maa-monthly}
\usepackage{epigraph}

\theoremstyle{theorem}
\newtheorem{theorem}{Theorem}

\newtheorem{lemma}[theorem]{Lemma}

\theoremstyle{definition}

\DeclareMathOperator{\cut}{Cut}

\mathchardef\mhyphen="2D

\def\semicolon{;}
\def\applytolist#1{
    \expandafter\def\csname multi#1\endcsname##1{
        \def\multiack{##1}\ifx\multiack\semicolon
            \def\next{\relax}
        \else
            \csname #1\endcsname{##1}
            \def\next{\csname multi#1\endcsname}
        \fi
        \next}
    \csname multi#1\endcsname}

\def\calc#1{\expandafter\def\csname c#1\endcsname{{\mathcal #1}}}
\applytolist{calc}QWERTYUIOPLKJHGFDSAZXCVBNM;
\def\bbc#1{\expandafter\def\csname bb#1\endcsname{{\mathbb #1}}}
\applytolist{bbc}QWERTYUIOPLKJHGFDSAZXCVBNM;
\def\bfc#1{\expandafter\def\csname bf#1\endcsname{{\mathbf #1}}}
\applytolist{bfc}QWERTYUIOPLKJHGFDSAZXCVBNM;
\def\sfc#1{\expandafter\def\csname s#1\endcsname{{\sf #1}}}
\applytolist{sfc}QWERTYUIOPLKJHGFDSAZXCVBNM;
\def\fc#1{\expandafter\def\csname f#1\endcsname{{\mathfrak #1}}}
\applytolist{fc}QWERTYUIOPLKJHGFDSAZXCVBNM;

\begin{document}

\title{Asymptotics of Redistricting the $n\times n$ Grid}
\markright{Asymptotics of redistricting}
\author{Christopher Donnay and Matthew Kahle}  
\maketitle

\begin{abstract}
Redistricting is the act of dividing a region into districts for electoral representation. Motivated by this application, we study two questions.
How many ways are there to partition the $n\times n$ grid into $n$ contiguous districts of equal size?
How many of these partitions are ``compact"?
We give asymptotic bounds on the number of plans: a lower bound of roughly $1.41^{n^2}$ and an upper bound of roughly $3.21^{n^2}$.
We then use the lower bound to show that most plans are not compact.
\end{abstract}

\epigraph{``So you've got---let's say you've got 100 maps or you might even have 25. 
I think you probably have thousands."}{Justice Alito in \textit{Oral Arguments for Rucho v. Common Cause 2019}}
\section{Introduction.}
\label{introduction}
Redistricting is the act of dividing a region into districts for electoral representation.
Mathematics has played an increasingly important role in redistricting in the U.S. in the past decade; see, for example, the amicus brief to the Supreme Court in the case Rucho v. Common Cause, 2019 \cite{math_amicus_brief}.
One of the first natural questions to arise in the mathematical study and litigation of redistricting
is, ``How many redistricting plans are there for a particular region?"
While this question may be impossible to answer precisely in real world examples (sorry Justice Alito!), we can make provable statements if we use a simplified model.

Consider the $n\times n$ grid.
Assume that each square of the grid represents a voter.
A redistricting plan is a geometric partition of the grid into $n$ districts where we require that each piece of the partition forms a connected region and that each district has exactly $n$ whole squares.\footnote{Connectivity and population balance are common real world redistricting requirements. We consider two squares to be connected if they share an edge of positive length, i.e., corners do not count. While exact population balance is required for U.S. Congressional districts, many state and local level districts allow for some imbalance, say up to about 5\%.}
The districts are \emph{$n$-ominoes}, shapes that can be made by gluing $n$ congruent squares edge to edge, so a redistricting plan is a tiling of the $n\times n$ grid by $n$-ominoes.

Let $\cP_n$ denote the set of tilings of the $n\times n$ grid by $n$-ominoes.
We are interested in studying $|\cP_n|$, the number of such tilings, as well as the structure of the tilings themselves.
These tilings are in bijection with certain partitions of the \emph{dual graph}, a graph made by assigning a vertex to each square of the grid, and connecting two vertices if their squares share an edge of positive length.
Then there is a 1-1 correspondence between partitions of the dual graph into $n$ disjoint, connected subgraphs with $n$ vertices and tilings of the $n\times n$ grid by $n$-ominoes.
It will be useful to pass back and forth between these perspectives; see Figure \ref{fig:dual_graph}.
\begin{figure}
    \centering
    \includegraphics{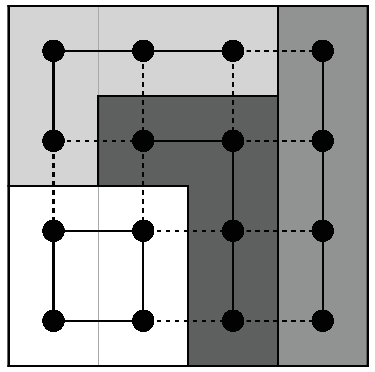}
    \caption{A tiling of the $4\times 4$ grid by tetrominoes and the corresponding dual graph and partition; dashed edges are cut in the partition. The cut score of this plan is 11.}
    \label{fig:dual_graph}
\end{figure}

Surprisingly little is known about $|\cP_n|$.
Only the values for $n=1,\dots,9$ have been computed, which we have reproduced in Table \ref{table: known values nxn} from  \cite{mggg_grid_partitions}.
\begin{table}[h]
\centering
\caption{The known values of $|\cP_n|$ as computed by the MGGG Redistricting Lab \cite{mggg_grid_partitions}.  The case of $n=10$ was too computationally intensive to be completed. This is OEIS sequence A172477. The original enumeration of the $9\times 9 $ grid can be found in \cite{harris_enumeration}. The same table from MGGG also has enumerations when you allow for population imbalance in the districts.\\}
\label{table: known values nxn}

\begin{tabular}{c|c}
$n$ & $|\cP_n|$        \\ \hline
2   & 2                   \\
3   & 10                  \\
4   & 117                 \\
5   & 4,006               \\
6   & 451,206             \\
7   & 158,753,814         \\
8   & 187,497,290,034     \\
9   & 706,152,947,468,301
\end{tabular}
\end{table}
Clearly the values of $|\cP_n|$ are undergoing combinatorial explosion.
There are over 700 trillion ways to tile the $9\times 9$ grid with 9-ominoes; there are more redistricting plans for the $25\times 25$ grid than there are atoms in the universe. 
Real world dual graphs are \emph{far} larger and \emph{far} more complex than the $25\times 25$ grid graph; the dual graph of Ohio (when divided into Census blocks) has over 365,000 vertices.
This should give some perspective on just how many possible plans there are in real world examples.
When studying redistricting plans we must be content with sampling instead.

In generating a sample, what distribution should we be sampling from?
It is tempting to say that we should sample redistricting plans uniformly from the space of all plans.
However, when done in practice, we observe that almost every plan that is generated is full of long, snakey districts.
See Figure \ref{fig:fractal_tiling}.
Not only that, uniform sampling itself is a difficult problem with a variety of obstructions \cite{frieze_subexp_mixing, najt2019complexitygeometrysamplingconnected}.\footnote{There are also interesting connections between sampling partitions and sampling non-intersecting lattice paths \cite{pegden2023directsamplingshortpaths, najt_emp_sampling}.}
\begin{figure}
    \centering
    \includegraphics[width=.45 \textwidth]{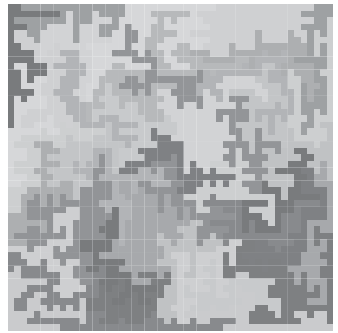}
    \caption{A ``typical" tiling of the $50\times 50$ grid is non-compact. This figure was generated using the Python package \texttt{gerrychain} and a swap Markov chain whose stationary distribution is uniform on $\cP_{50}$ \cite{gerrychain}. Here the fraction of cut edges is roughly 0.41. Lemma \ref{lemma: cut score bounds} gives a lower bound of about 0.144 and an upper bound of exactly 0.5.}
    \label{fig:fractal_tiling}
\end{figure}
These snakey plans would never be enacted, as they are not \emph{compact}.
This is not the usual meaning of compactness from topology.
In the context of redistricting, compact districts are districts that have a reasonable shape.
One of the authors' favorite non-compact districts is featured in Figure \ref{fig:goofy kicking}. 
\begin{figure}
    \centering
    \includegraphics{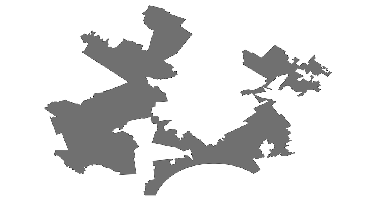}
    \caption{The ``Goofy Kicking Donald Duck" district, which was Pennsylvania's 7th Congressional district in the 2010 redistricting cycle \cite{wapo_goofy}. This district is not compact. The entire plan was struck down by the Pennsylvania Supreme Court as a partisan gerrymander in 2018.}
    \label{fig:goofy kicking}
\end{figure}
Despite its ``we'll know it when we see it" definition, compactness is a nigh universal redistricting requirement.
For a discussion of a sampling method that does favor compact districts, we refer readers to \cite{deford2019recombination}.

Traditionally, compactness has been measured using the Polsby--Popper score, which is just an eponymous recasting of the isoperimetric inequality \cite{polsbypopper}.
The Polsby--Popper score of a district $D$ with area $A(D)$ and perimeter $P(D)$ is defined as \[PP(D) = \frac{4\pi A(D)}{P(D)^2}.\]
This is a ratio between 0 and 1, with 1 being achieved by the circle.
Notably, the Polsby--Popper score is for a single district.
Rather than take some summary statistic over the set of districts in a plan,
we instead use the \emph{cut score}, which is a plan-wide score, as explained in \cite{compactness_explainer}.\footnote{A review of the history of cut scores and cut sets can be found in Section 5.4 of \cite{duchin2023discrete}. This is also related to the Cheeger constant, which plays an important role in spectral graph theory.}
Let $\cut(P)$, the cut score of a partition $P$, denote the number of edges $\{u,v\}$ such that $u$ is in one district and $v$ is in another. 
Then $\cut(P)$ is a discrete measure of the compactness of a redistricting plan; a higher cut score indicates elongated boundaries between districts. 
For an example of this computation, refer back to Figure \ref{fig:dual_graph}.

We will give asymptotic bounds on the rate of growth of $|\cP_n|$: a lower bound of roughly $1.41^{n^2}$ and an upper bound of roughly $3.21^{n^2}$.
We compare our upper and lower bounds to the known values in Figure \ref{fig:log plot}.
Our proofs are elementary and nearly self-contained. 
We will then use our lower bound to prove that most redistricting plans are not compact.
In other words, a randomly selected districting plan will have a high cut score, thus giving rigor to the observation that uniform sampling only produces non-compact districts.

 \begin{figure}
        \centering
        \includegraphics{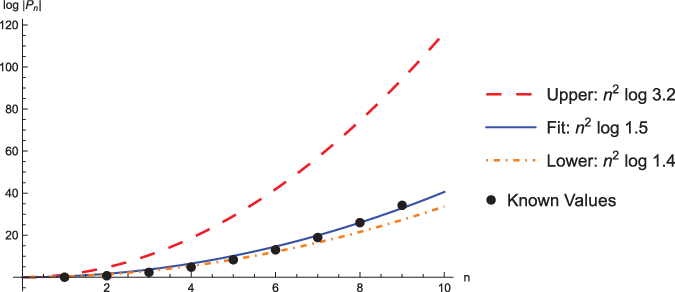}
        \caption{The known values of $\log |\cP_n|$ for $n=1,\dots,9$, along with our upper and lower bounds. The solid line is a quadratic fit of the data points.}
        \label{fig:log plot}
    \end{figure}

\section{Asymptotic Growth of $|\cP_n|$.}
In discussing the asymptotics of $|\cP_n|$, we will make use of the relations given in Table \ref{tab:asymptotic notation}.

\begin{table}[]
    \centering
    \caption{Notations for asymptotic relations used in this article.\\}
    \label{tab:asymptotic notation}
    \begin{tabular}{c|c|c}
    Relation & Example & Formal Definition \\ \hline \rule{0pt}{15pt}
        $f\approx g$ & $3n^2\approx n^2$ & $0< \liminf_{n\to\infty} \frac{f(n)}{g(n)}\le\limsup_{n\to\infty} \frac{f(n)}{g(n)}<\infty  $  \\ \rule{0pt}{15pt}
        $f\sim g$ & $n!\sim \sqrt{2\pi n} \left(\frac{n}{e}\right)^n$ & $\lim_{n\to\infty} \frac{f(n)}{g(n)} = 1$\\ \rule{0pt}{15pt}
        $f \lesssim g$ & $\frac{1}{n} \lesssim 1$ & $\limsup_{n\to\infty} \frac{f(n)}{g(n)}<\infty$\\ \rule{0pt}{20pt}
        $f \ll g$ & $n \ll n^2$ & $\lim_{n\to\infty} \frac{f(n)}{g(n)} = 0$\\ \rule{0pt}{15pt}
            $f=o(1)$ & $f(n)=\frac{1}{\log(n)}$ & $\lim_{n\to\infty} f(n) = 0 $\\
    \end{tabular}
    
\end{table}
For our upper bound, it will be convenient to talk in the language of partitions.\footnote{In the real world, districting plans are labeled, but often in ways that are difficult to account for. We choose to use unlabeled partitions, but readers should note this affects our counts.}
Let $G_n$ denote the dual graph of the $n\times n$ grid, which is itself an $n\times n$ grid graph.
A trivial upper bound on the number of partitions of any graph is $2^{|E(G)|}$,  since any partition can be described by which edges of the graph are on or off as in Figure \ref{fig:dual_graph}.
In the case of the $n\times n$ grid graph, $|E(G_n)|=2n(n-1)$,  and thus a trivial upper bound is \[
|\cP_n| \le 2^{2n(n-1)} = 4^{n^2-n}= (4^{1-\frac{1}{n}})^{n^2}= (4-o(1))^{n^2}.\] 
Any partition of $G_n$ into $n$ pieces can be constructed by drawing a spanning tree on the vertices of $G_n$ and cutting $n-1$ edges of the tree.
See Figure \ref{fig:spanning_tree_cut}.
\begin{figure}
    \centering
    \includegraphics[width=.5 \textwidth]{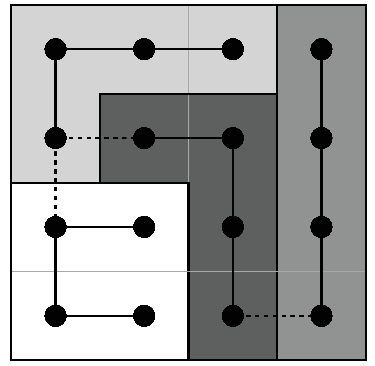}
    \caption{One of the spanning trees that could have generated the partition in Figure \ref{fig:dual_graph}. The cut edges are dashed.}
    \label{fig:spanning_tree_cut}
\end{figure}
The asymptotics of the number of spanning trees on the square grid were proved over a series of papers, culminating in \cite{wu1977number}. We have that \[\lim_{n\to\infty} \frac{\log \tau(G_n)} {n^2} = \frac{4C}{\pi},\] where $\tau(G_n)$ is the number of spanning trees of the $n\times n$ grid graph, and $C$ is Catalan's constant.\footnote{Catalan's constant is defined as $C:= \sum_{n=0}^\infty \frac{(-1)^n}{(2n+1)^2}$, and appears in topology, combinatorics, and statistical mechanics.}
Let $\mathfrak{b}:= \exp\left(\frac{4 C}{\pi}\right)= 3.2099\dots$.
Then a better upper bound for $|\cP_n|$ is the following.
\begin{theorem}[An Upper Bound]
\label{thm:upper bound}
$|\cP_n|\lesssim(\mathfrak{b}+o(1))^{n^2}$. 
\end{theorem}

\begin{proof}
By the argument above,  $|\cP_n|\le \tau(G_n) {n^2-1\choose n-1}$.
We have that $\tau(G_n) = (\mathfrak{b}+o(1))^{n^2}$.
Moreover, 
\begin{align*}
{n^2-1\choose n-1} & \le \frac{(n^2-1)^{n-1}}{(n-1)!} \ll \frac{(n^2)^n}{(n-1)!}\\
& \sim n(n^2)^n \frac{1}{\sqrt{2\pi n} \left(\frac{n}{e}\right)^n} \quad \text{(Stirling's approximation)}\\
& \approx (en)^n \sqrt{n} = (en^{1+\frac{1}{2n}})^n=(en+o(1))^n.
\end{align*}
Thus,
\[  |\cP_n| \le \tau(G_n) {n^2-1\choose n-1}\lesssim(\mathfrak{b}+o(1))^{n^2}(en+o(1))^n= (\mathfrak{b}+o(1))^{n^2}.\qedhere\]
\end{proof}

To prove a lower bound, it suffices to construct some number of tilings of the $n\times n$ grid.
An earlier version of this paper had a lower bound of $(3^{1/6}-o(1))^{n^2}$.
Thanks to a proof suggestion from Jamie Tucker-Foltz, we were able to improve the bound to the following.

\begin{theorem}
$|\cP_n|\gtrsim (2^{1/2}+o(1))^{n^2}$.
\end{theorem}
\label{thm: lower bound}

\begin{proof}
Our general strategy will be to fix a partial tiling of the grid, and then complete the tiling in a large number of ways.
As we complete the tiling, we ensure that the districts are always contiguous and of equal size.

Assume first that $n\equiv 0 \pmod{4}$.
Partially tile the grid as in Figure \ref{fig:n over 2 tiling}.
Across the top row of the grid,  starting in the first column, place a tile of size $\frac{n}{4}\times 1$ in every other column, giving each tile a unique district color.
Place one more tile in the final column, assigning it to the same district as the penultimate column's tile.
Then in row $\frac{n}{4}+1$, do the same, but place the tiles in the columns you skipped before. 
Assign them to the same district as the tile that is queen adjacent to their upper left corner.
Place one more tile of size $\frac{n}{4}\times 1$ in the first column at row $\frac{n}{4}+1$, and give it the same color as the district already in that column.
With different colors,  identically tile the bottom half of the grid.
Leave the remaining squares of the grid untiled.
The other three congruence cases, $n\equiv 1,2,3\pmod{4}$,  are an essentially identical argument with the same asymptotics. 
See Figure \ref{fig:n over 2 tiling other congruence}.

\begin{figure}
        \centering
        \includegraphics{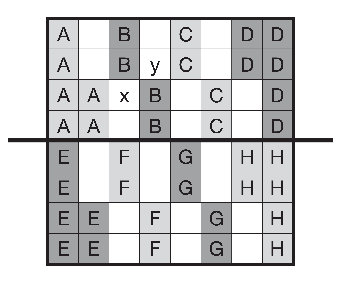}
        \includegraphics{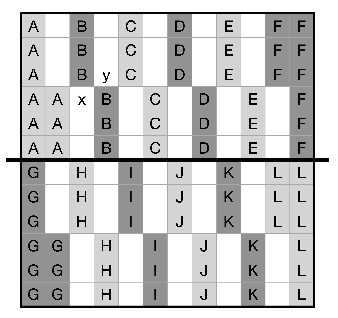}
        \caption{The initial partial tiling of the $n\times n$ grid with $n=8$ and $n=12$. Note that in this figure, each capital letter represents a unique district; that is, even though the two sets of B squares are discontiguous, they both represent the same district, and will be made to be contiguous later.  White denotes squares that have yet to be assigned a district.}
        \label{fig:n over 2 tiling}
    \end{figure}

Now we wish to complete the partial tiling in as many ways as possible.
Our end goal is to have $n/2$ districts in the top half of the grid and $n/2$ in the bottom. 
\begin{figure}
        \centering
        \includegraphics{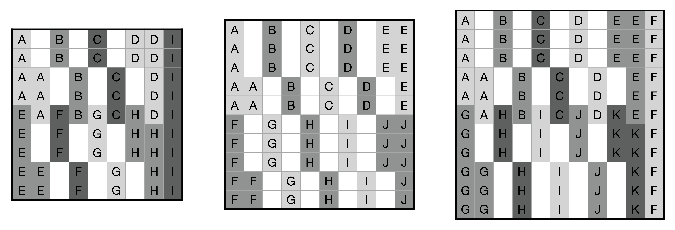}
        \caption{The other initial partial tilings for $n\equiv 1,2,3\pmod{4}$.  Note that in this figure, each letter represents a unique district; that is, even though the two sets of B squares are discontiguous, they both represent the same district, and will be made to be contiguous later.  White denotes squares that have yet to be assigned a district.}
        \label{fig:n over 2 tiling other congruence}
    \end{figure}
Choose $n/4$ of the $n/2$ white squares to the right of district A to assign to district A. 
Assign the remaining $n/4$ to district B.
Repeat this process moving across the grid.
An example of a completed tiling is given in Figure \ref{fig:example_n_over_2}.
\begin{figure}
        \centering
        \includegraphics{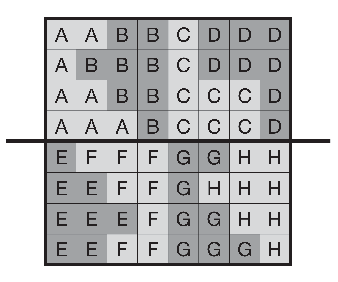}
        \caption{A completed tiling of the $8\times 8$ grid after assigning the white squares.}
        \label{fig:example_n_over_2}
    \end{figure}
There are 
\[{n/2\choose n/4}^{n-2}\]
ways to assign these white squares.
However, not all assignments create contiguous districts.
We must have that at least one of the squares $x,y$ are assigned to district B (and likewise for districts C,  F,  and G in the $8\times 8$ example of Figure \ref{fig:n over 2 tiling}).
It is easier to count the fraction of assignments in which both $x$ and $y$ are \emph{not} assigned to district B; this is computed as 
\[\left(\frac{{n/2-1\choose n/4}}{{n/2\choose n/4}}\right)^2 = \frac{1}{4}.\]
Hence
\[|\cP_n|\ge \left(\frac{3}{4}\right)^{n-4}{n/2\choose n/4}^{n-2}.\]

\clearpage
Now it is just a matter of computing the asymptotics of the above formula.
By Stirling's approximation,
\begin{align*}
|\cP_n| &\ge \left(\frac{3}{4}\right)^{n-4}{n/2\choose n/4}^{n-2}\\
& \approx  \left(\frac{3}{4}\right)^{n-4} \left(2^{n/2}\sqrt{\frac{4}{\pi n}}\right)^{n-2}\\
& = (2^{1/2}+o(1))^{n^2}.\qedhere
\end{align*}
\end{proof}

\section{Most Redistricting Plans Are Not Compact.}
In order to measure the compactness of a redistricting plan, we will use the cut score defined in Section \ref{introduction}.
We note that a similar result holds if you instead frame your compactness bound in terms of an isoperimetric constraint.
We first observe that there are tight bounds on the cut score of a plan $P\in \cP_n$. 
\begin{lemma}
\label{lemma: cut score bounds}
For all $P\in \cP_n$, we have that $2n^{3/2}-2n\le \cut(P)\le n(n-1)$.
\end{lemma}

\begin{proof}
The upper bound comes from the observation that the plans with the highest cut scores are when the dual graph of each district is a tree.
For the lower bound, let $P_i$ denote the $i$th district in plan $P$, and let $|\partial P_i|$ denote the length of the perimeter of $P_i$.
Since \[\sum_{i=1}^n |\partial P_i| = 2\cut(P) +4n,\] and $|\partial P_i|\ge 2\lceil 2\sqrt{n}\rceil$ by \cite{extremal_stats_polyominoes}, we have $\cut(P) \ge 2n^{3/2}-2n$.
This lower bound is achieved by the plan where each $n$-omino is a $\sqrt{n}\times\sqrt{n}$ square.
See Figure \ref{fig:cut scores}.
\end{proof}

\begin{figure}
    \centering
    \includegraphics[width=.9 \textwidth]{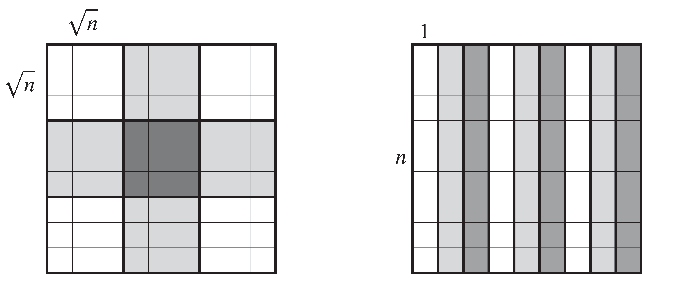}
    \caption{On the left,  a very compact tiling of the grid by squares, with cut score roughly $n^{3/2}$. As a least compact example, the tiling on the right by $n\times 1$ rectangles has cut score $n(n-1)$.}
    \label{fig:cut scores}
\end{figure}

We show below that a typical redistricting plan is non-compact.
More precisely, a typical redistricting plan has a cut score on the order of $n^2$, i.e., a constant fraction of the highest cut score possible according to Lemma \ref{lemma: cut score bounds}.
For $\varepsilon>0$, let $\cC_n(\varepsilon)\subseteq \cP_n$ denote the subset of partitions $P$ such that $\cut(P)\le \varepsilon n^2$.

\clearpage
\begin{theorem}
Let $0<\varepsilon<0.08$. Then $\lim_{n\to\infty} \frac{|\cC_n(\varepsilon)|}{|\cP_n|}  =  0$.
\end{theorem}

\begin{proof}
Since a partition is determined by its cut edges, our compactness assumption implies that 
\[|\cC_n(\varepsilon) |\le \sum_{j = n-1}^{\varepsilon n^2} {2n(n-1)\choose j}.\]
Note that we need at least $n-1$ cuts to make $n$ districts, and hence our sum begins indexing from $n-1$.
Then asymptotically we have 
\begin{align*}
\sum_{j = n-1}^{\varepsilon n^2} {2n(n-1)\choose j} & \le \sum_{j = n-1}^{\varepsilon n^2} {2n(n-1)\choose \varepsilon n^2} \quad (\text{$\varepsilon n^2 \le n(n-1)$ so binomial coefficients are increasing})\\
&\le n^2 {2n(n-1)\choose \varepsilon n^2} \ll  n^2 \frac{(2n^2)^{\varepsilon n^2}}{(\varepsilon n^2)!}\\
& \sim \frac{n}{\sqrt{2\pi \varepsilon}} \left(\frac{2e}{\varepsilon}\right)^{\varepsilon n^2}\quad(\text{Stirling's approximation})\\
& = \left(\left(\frac{2e}{\varepsilon}\right)^{\varepsilon}+o(1)\right)^{n^2}.
\end{align*}

Thus in order for $\lim_{n\to\infty} \frac{|\cC_n(\varepsilon)|}{|\cP_n|}  =  0$, by Theorem \ref{thm: lower bound} it suffices that $\left(\frac{2e}{\varepsilon}\right)^{\varepsilon} < 2^{1/2}$.
This occurs for sufficiently small $\varepsilon$  ($\varepsilon<.08$ is sufficient).\qedhere
\end{proof}

\section{Conclusion.}
We have shown that $|\cP_n|$ grows exponentially in $n^2$, and the base of our lower and upper bound differ by about 1.8. 
We have also shown that a typical plan is not compact, with cut score of the highest order possible.
Thus, there are many redistricting plans and most are not compact.

The proof of our upper bound generalizes quite easily to the case of $k \ll n^2$ districts, as well as population imbalance in the districts.
In fact, our argument in the proof of Theorem \ref{thm:upper bound} explicitly ignored the population balance constraint.
We imagine that the proof of the lower bound (and thus the compactness result) generalizes as well, with a bit more work to define the initial partial tilings, when the number of districts $k$ is small, perhaps $k=o(n^{3/2})$.

There is still a lot that we do not know.
For example,  is $|\cP_n|$ increasing in $n$? 
There is no clear injection from $\cP_n\to \cP_{n+1}$ due to the change in the size of the tiles.
Some other interesting open problems include:
\begin{enumerate}
\item Can the upper bound be improved? The only way to improve our current argument is to get a handle on the number of spanning trees which can be cut into balanced pieces. Recent progress on the question of cutting spanning trees into equal pieces can be found in \cite{cannon2024sampling}.

\item Can the lower bound be extended to districts with imbalance in size?

\item Does $\lim_{n\to\infty} |\cP_n|^{1/n^2}$ exist? 
This would follow immediately from Fekete's lemma if one could show that $|\cP_m| \times |\cP_n| \le |\cP_{m+n}|$.

\item Can the compactness results be extended to other compactness scores, such as the spanning tree score of \cite{duchin2023discrete}?

\item What can be said for more general dual graphs $G$, like subsets of the grid, non-square grids, or subsets of the triangular lattice?\footnote{The triangular lattice has also been proposed as a good abstraction of real world dual graphs, since these graphs frequently have many triangles.} 
We expect that in all scenarios, the number of partitions should be exponential in the size of the graph, with the same proof techniques going through.
\end{enumerate}

\section{Acknowledgments.} 
The first author would like to thank Moon Duchin and the second author for co-advising him.
The authors would like to thank Moon Duchin, Carlos Mart\'inez, Dustin Mixon, Jamie Tucker-Foltz, the Editorial Board, and the two anonymous reviewers for helpful conversations and suggestions that improved the quality of this paper.
This material is based upon work supported by the National Science Foundation under Grant No. DMS-1928930 and by the Alfred P. Sloan Foundation under grant G-2021-16778, while the first author was in residence at the Simons Laufer Mathematical Sciences Institute (formerly MSRI) in Berkeley, California, during the Fall 2023 semester. The authors gratefully acknowledge NSF-DMS \#1547357 and \#2005630.

\bibliographystyle{vancouver}
\bibliography{main.bbl}

\begin{biog}
\item[Christopher Donnay] Chris(topher) Donnay is a PhD candidate in mathematics at \emph{The} Ohio State University.  His interests are in stochastic topology, redistricting, and computational social choice.  As a former high school math and computer science teacher, Chris is also passionate about pedagogy and science communication.
\begin{affil}
Department of Mathematics, The Ohio State University\\
donnay.1@osu.edu
\end{affil}
\end{biog}

\begin{biog}
\item[Matthew Kahle] Matthew Kahle has been faculty at \emph{The} Ohio State University since 2011.  His mathematical interests include various interactions of topology and geometry with combinatorics, probability, and statistical physics. Outside of mathematics, he enjoys spending time with his family, cooking, and bicycle commuting.

\begin{affil}
Department of Mathematics, The Ohio State University\\
\end{affil}
\end{biog}

\vfill\eject

\end{document}